\documentclass{article}

\usepackage{amsmath}
\usepackage{amssymb}
\usepackage{amsthm}
\usepackage{comment}
\usepackage{tikz-cd}
\usepackage{bbm}
\usepackage{amscd}
\usepackage{hyperref}
\usepackage{xparse}
\usepackage{sseq}
\usepackage{float}
\usepackage[makeroom]{cancel}
\usepackage[margin=1.1in]{geometry}
\usepackage{fancyhdr}
\usepackage[backend=bibtex,style=alphabetic]{biblatex}
\usepackage{subfig}

\newtheorem{theorem}{Theorem}[section]
\newtheorem{lemma}{Lemma}
\newtheorem{proposition}{Proposition}
\newtheorem{corollary}{Corollary}

\newtheorem{remark}{Remark}

\DeclareMathOperator{\im}{im}

\everymath{\displaystyle}

\newcommand{\mb}[1]{\mathbb{#1}}
\newcommand{\mf}[1]{\mathfrak{#1}}
\newcommand{\mc}[1]{\mathcal{#1}}
\newcommand{\tx}{\text}

\newcommand{\ti}{\textit}

\newcommand{\pmtx}[1]{\begin{pmatrix} #1 \end{pmatrix}}

\makeatletter
\newcommand{\blk}[3][\@nil]{
 \def\tmp{#1}
   \ifx\tmp\@nnil
       \begin{#2}
       		#3
       \end{#2}    	
    \else
       \begin{#2}[#1]
       		#3
       \end{#2}	
    \fi
}
\makeatletter

\NewDocumentEnvironment{diagram}{O{}}
  {\begin{center}\begin{tikzcd}[#1]}
  {\end{tikzcd}\end{center}}
  
\usetikzlibrary{matrix}

\usepackage[bottom]{footmisc}

\title{The Conley-Zehnder Index of Brownian Paths on $\tx{Sp}(2, \mb{R})$}
\author{Yuchen Fu}

\begin{document}

\maketitle

\begin{abstract}
	We investigate the probability distribution of Conley-Zehnder indices associated with Brownian random paths on $\tx{Sp}(2n, \mb{R})$ that start at the identity. In the case of $n = 1$, we prove that the distribution has the same moment asymptotics as the standard random walk on the real line. We also present numerical evidence suggesting that the same asymptotics should hold for general $n$.
\end{abstract}

\section{Introduction}
Let $\tx{Sp}(2n, \mb{R})$ denote the group of linear symplectomorphisms $\mb{R}^{2n} \to \mb{R}^{2n}$, which can also be considered as the matrix group $\{A \in \tx{GL}(2n, \mb{R}) \mid A^t J_0 A = J_0\}$ for $J_0 = \pmtx{0 & I_n \\ - I_n & 0}$ under a fixed basis. Let $(\phi_t)_{0 \le t \le 1}$ be a continuous path on $\tx{Sp}(2n, \mb{R})$ such that $\phi_0 = I$ and $\det(\phi_1 - I) \neq 0$. To each such path we can assign an integer-valued Maslov-type index called the \ti{Conley-Zehnder index} (see Section~\ref{sec:CZ-index} for the full definition).

The Lie group structure of $\tx{Sp}(2n, \mb{R})$ allows us to define Brownian motion on the group (see Section~\ref{sec:random-walk} for definitions), which almost always has continuous sample paths that have well-defined Conley-Zehnder indices. In this way we have a probability distribution on $\mb{Z}$ for the Conley-Zehnder indices of these paths. Let $X$ be a random variable having such a probability distribution, where the Brownian motion has difffusitivity $c$. Let $\mb{E}$ denote the expectation of random variables. In this paper, we prove that for $n = 1$, we have the following asymptotics:

\begin{theorem}
	$\mb{E}(X^{2k + 1}) = 0$, $\mb{E}(X^{2k}) = \Theta(c^k)$ for each $k \in \mb{N}$.
\end{theorem}
Note that this is the same asymptotics as Brownian motion on the real line. We also expect this to hold for general $n$; for some numerical evidence, see the Appendix.

The starting point of the proof is Theorem~\ref{thm:morse-index}, which says that the Conley-Zehnder index is equal to the signature of some large finite-dimensional Hessian matrix. This allows us to convert the problem into studying the eigenvalue distribution of some random matrices, which in turn can be turned into studying the evolution of a single point under randomized Mobius transformations. We describe the continuous limit of this evolution as the solution of some SDE, and obtain the result from a heat-kernel-type estimate on a parabolic PDE that describes the transition probability of its solution.

\subsection{Acknowledgements}
The author thanks Prof. Paul Seidel for proposing the problem and sharing many critical insights, and Umut Varolgunes for mentoring the project and offering numerous inspiring discussions and helpful suggestions throughout the program. The author also thanks the MIT UROP+ program for the opportunity to conduct this project.

\newpage

\section{Discrete Approximation of Action Integral}
\label{sec:action-integral}
We follow the presentation of \cite{mcduff1998introduction}. Let $M$ be a symplectic manifold, with $\omega$ the (closed, nondegenerate) symplectic form. Let $(H_t)_{0 \le t \le 1}: M \to \mb{R}$ be a time-dependent Hamiltonian on $M$. Then the corresponding Hamiltonian vector field $X_t$ is defined by $\iota(X_t) \omega = d H_t$. The flow of this vector field, which is denoted by $(\phi_t): M \to M$ and is defined by the following equation
\[\frac{\partial}{\partial t} \phi_t = X_t \circ \phi_t, \phi_0 = \tx{id}\]
is called a \ti{Hamiltonian isotopy}. It's clear that $(\phi_t)$ are symplectomorphisms of $M$. We say a symplectomorphism $\phi$ is a \ti{Hamiltonian symplectomorphism} if there exists a Hamiltonian isotopy $(\phi_t)$ such that $\phi_1 = \phi$.

Conversely, suppose we have an arbitrary \ti{symplectic isotopy}, i.e. a smooth family of symplectomorphisms $(\phi_t)_{0 \le t \le 1}: M \to M$ such that $\phi_0 = \tx{id}$, then there exists a symplectic vector field $(X_t)$ given by the same equation above, and thus $d \iota(X_t) \omega = 0$ for all $t$. If the first de Rham cohomology group vanishes (which is the case, for instance, when $M$ is simply connected), then there always exists a smooth family of Hamiltonians $(H_t)$ such that $\iota(X_t) \omega = d H_t$.

Let us fix a particular $\lambda \in \Omega^1(M)$ such that $d \lambda = - \omega$. For each time-dependent Hamiltonian $(H_t)$ with the associated Hamiltonian isotopy $(\phi_t)$ we can define the \ti{action integral}
\[\mc{A}_{H}(z) = \int_{0}^{1}(\lambda(\frac{\partial}{\partial t} \phi_t(z))) - H_t(\phi_t(z)) dt\]

For any symplectomorphism $\phi: M \to M$, we can define its graph $\tx{gr}_\phi: M \to M \times M$ given by $\tx{gr}_\phi(x) = (x, \phi(x))$. Let $\alpha \in \Omega^1(M \times M)$ be any 1-form such that $-d \alpha = (-\omega) \oplus \omega$ and $\alpha|_\Delta = 0$ where $\Delta$ denotes the diagonal.

\begin{lemma}[\cite{mcduff1998introduction}, Corollary 9.21]
	If $\phi$ is a Hamiltonian symplectomorphism, then $\tx{gr}_\phi^* \alpha \in \Omega^1(M)$ is exact.
\end{lemma}

A function $S_{\alpha, \phi}: M \to \mb{R}$ such that $d S_{\alpha, \phi} = \tx{gr}_\phi^* \alpha$ is called an \ti{$\alpha$-generating function} for $\phi$. For $\phi$ a Hamiltonian symplectomorphism, let $(H_t)$ be the associated Hamiltonian, then we have $S_{(-\lambda) \oplus \lambda, \phi} = \mc{A}_H$.

From now on, we focus on the specific case $M = \mb{R}^{2n}$ and $\omega = J_0 = \pmtx{0 & I_n \\ -I_n & 0}$. We write $(q_1, q_2, p_1, p_2)$ as coordinates for $T^* \mb{R}^{2n}$, where $q_1, q_2 \in \mb{R}^n$ are the base coordinates and $p_1, p_2 \in \mb{R}^n$ are the fibre coordinates. Then the canonical primitive can be written as $\lambda = p_1 d q_1 + p_2 d q_2$. In this setting, there exists a global linear symplectomorphism $\Psi: \mb{R}^{2n} \times \mb{R}^{2n} \to T^* \mb{R}^{2n}$ that sends the diagonal to the zero section in the cotangent bundle, given as follows:
\[\Psi(x_0, y_0, x_1, y_1) = (x_1, y_0, y_1 - y_0, x_0 - x_1)\]
where $x_i, y_i \in \mb{R}^n$. The pullback of the canonical 1-form under this map is given by $\Psi^* \lambda = (y_1 - y_0) d x_1 + (x_0 - x_1) d y_0$. Now suppose $\phi: \mb{R}^{2n} \to \mb{R}^{2n}$ is a symplectomorphism, and write $\phi(x_0, y_0) = (x_1, y_1)$. Suppose $\phi$ is sufficiently $\mc{C}^1$-close to the identity such that there exists a function $G$ such that $G(x_1, y_0) = x_0$ iff $\phi(x_0, y_0) = (x_1, y_1)$ for some $y_1$. Define $V$ by $V(x_1, y_0) = S_{\Psi^* \lambda, \phi}(G(x_1, y_0), y_0)$, then we have
\begin{equation}
\label{eq:typeVgen}
x_1 - x_0 = \frac{\partial V}{\partial y_0}(x_1, y_0) \hspace{1em} y_1 - y_0 = - \frac{\partial V}{\partial x_1}(x_1, y_0)
\end{equation}
We shall refer to a smooth function $V$ such that Equation~\ref{eq:typeVgen} holds as a \ti{generating function of type V}, or simply a generating function of $\phi$.

\begin{lemma}
If we write a linear symplectomorphism $\phi$ in the block form $\pmtx{A & B \\ C & D}$, then it admits a generating function of type $V$ iff $\det(A) \neq 0$, and a generating function is given by
\[V(x_1, y_0) = -\frac{1}{2} \langle x_1, C A^{-1} x_1 \rangle + \langle y_0, (\tx{I}_n - A^{-1}) x_1 \rangle + \frac{1}{2} \langle y_0, A^{-1} B y_0 \rangle\]
where $\langle \cdot, \cdot \rangle$ is the usual dot product on $\mb{R}^n$.
\end{lemma}

\begin{proof}
	Given that $\det(A) \neq 0$, the formula above gives a well-defined function $V$. Recall that $\phi$ being symplectic implies that $A B^t$ is symmetric, which implies $A^{-1} B$ is symmetric, so $\sum_{i = 1}^{n} \langle e_i, A^{-1} B y_0 \rangle + \langle y_0, A^{-1} B e_i \rangle = 2 A^{-1} B y_0$ for basis vectors $e_i$, so we have $\frac{\partial V}{\partial y_0}(x_1, y_0) = (I_n - A^{-1}) x_1 + A^{-1} B y_0 = x_1 - x_0$ because $x_1 = A x_0 + B y_0$. The same goes for the $x_1$ derivative. Conversely, suppose $\phi$ has a generating function $V(x_1, y_0)$, then it is possible to determine $x_0$ from $x_1$ and $y_0$ alone; however, if $A x' = 0$ for some nonzero vector $x'$, then $\phi(x_0, y_0) = (x_1, y_1) \implies \phi(x_0 + a x', y_0) = (x_1, y_1 + a C x')$ for any $a \in \mb{R}$, contradiction.
\end{proof}

Now again let $(\phi_t)$ be a Hamiltonian isotopy. For any $N \in \mb{N}^+$, we can define the \ti{discrete isotopies} $(\psi_j)_{0 \le j < N}$ given by $\psi_j = \phi_{\frac{j + 1}{N}} \phi^{-1}_{\frac{j}{N}}$. For $N$ large enough, all these maps will be close to identity in the sense that they admit corresponding generating functions $(V_j)$. Now for $z = (x_0, \ldots, x_N, y_0, \ldots, y_{N - 1})$ where $x_i, y_i \in \mb{R}^n$ we shall define the \ti{discrete action} as
\[\Phi(z) = \sum_{j = 0}^{N - 1}\langle y_j, x_{j + 1} - x_j \rangle - V_j(x_{j + 1}, y_j)\]
As $N \to \infty$, this value will converge to $\mc{A}_H(z)$ for $\lambda = \langle y, d x \rangle$.

\section{The Conley-Zehnder Index}
\label{sec:CZ-index}
Again our setting is $(M, \omega) = (\tx{Sp}(2n, \mb{R}), J_0)$. Let $\phi = (\phi_t)$ be a linear symplectic isotopy, i.e. $\phi_t \in \tx{Sp}(2n, \mb{R})$ for each $t$. One can of course think of $\phi$ as a path on $\tx{Sp}(2n, \mb{R})$. Now consider the polar decomposition $A = U_A P_A$ where $U_A$ is the unitary part and $P_A$ is the positive semidefinite part. The map $A \mapsto U_A$ is then a continuous retraction $\tx{Sp}(2n, \mb{R}) \to \tx{U}(n, \mb{C})$. Let $\textstyle\det_\mb{C}$ denote the complex determinant, then it is a map from $\tx{U}(n, \mb{C})$ to the unit circle. Denote the composition of these two maps by $\rho: \tx{Sp}(2n, \mb{R}) \to S^1$. Then clearly $\rho(\phi)$ is a path on the unit circle starting at $1$. If $\phi$ is a loop (i.e. $\phi_0 = \phi_1$), then we can define its \ti{Maslov index} $\mu(\phi) = \deg(\rho(\phi))$.

Now suppose $\phi$ is a general path that starts at the identity, and suppose that $\phi_1$ does not land on $\tx{Sp}_0 = \{X \in \tx{Sp}(2n, \mb{R}) \mid \det(X - I) = 0\}$. Let's call such path \ti{admissible}.

\begin{lemma}[\cite{salamon1997lectures}]
\label{lemma:conley-zehnder-def}
	To each admissible path one can assign an integer-valued \ti{Conley-Zehnder index} $\mu_{CZ}(\phi)$ that has the following properties:
	\begin{itemize}
	\item $\mu_{CZ}(\phi)$ is invariant under homotopy of admissible paths.
	\item If $\phi_0 = \phi_1 = \tx{id}$, then for any other admissible path $\psi$, we have $\mu_{CZ}(\phi \psi) = \mu_{CZ}(\psi) + 2 \mu(\phi)$.
	\item If $S$ is a symmetric non-degenerate matrix with all eigenvalues of absolute value less than $2 \pi$, and if $\phi_t = \exp(J_0 S t)$, then $\mu_{CZ}(\phi) = \frac{1}{2} \tx{sign}(S)$, where $\tx{sign}$ is the signature (i.e. number of positive eigenvalues minus that of negative eigenvalues).
	\end{itemize}
	Moreover, these properties uniquely characterize the Conley-Zehnder index.
\end{lemma}

The computation of the index can be done as follows. The complement of $\tx{Sp}_0$ consists of two parts $\tx{Sp}_+ = \{X \mid \det(X - I) > 0\}$ and $\tx{Sp}_- = \{X \mid \det(X - I) < 0\}$, and these two parts are respectively path connected. Given a path $\phi$, first choose an extension $\tilde{\phi}$ that ends at either $W_+ = -I \in \tx{Sp}^+$ or $W_- = \tx{diag}(2, 1/2, -1, \ldots, -1) \in \tx{Sp}^-$ within the component that $\phi_1$ lies in. By homotopy property, the choice of this path does not matter. The Conley-Zehnder index is then given by $\deg(\rho^2(\tilde{\phi}))$.

Since our base manifold is $\mb{R}^{2n}$, $(\phi_t)$ is Hamiltonian, so we can define the corresponding discrete action $\Phi(z)$ as described in the last section. Let $d^2 \Phi(z)$ denote the Hessian of the discrete action; notice that in our current setting, this is a symmetric bilinear form that is independent of where the partial derivatives are evaluated. A simple calculation shows that the Hessian is nondegenerate if and only if $\det(\phi_1 - I) \neq 0$. The following constitutes the basis of our work in this paper:

\begin{theorem}[\cite{mcduff1998introduction}, Remark 9.18]
\label{thm:morse-index}
	Suppose $d^2 \Phi(z)$ is nondegenerate. Then as $N \to \infty$, the signature of $d^2 \Phi(z)$ converges to $\mu_{CZ}(\phi)$.
\end{theorem}

\section{Random Walk on Lie Group}
\label{sec:random-walk}
Let $G$ be a connected Lie group, equipped with some canonical left-invariant Haar measure $\mu_{\tx{Haar}}$. Let $(\mu_n)_{n \ge 1}$	be a sequence of probability measures on $G$, and let $(X_{k, n})_{1 \le k \le n}$ be iid random variables with law $\mu_n$. Define $S_{j, n} = X_{1, n} \ldots X_{j, n}$, so they have law $\mu_n^{* j}$.
Define a random process $(X^n_t)_{0 \le t \le 1}$ by $X^n_t = S_{j, n}$ if $j \le \lfloor n t \rfloor < j + 1$. This will be used to approximate of a random path on $G$ as we take $n \to \infty$.

Fix a basis $(e_i)$ for the associated Lie algebra $\mf{g}$, and define continuous functions $x_i: U_0 \to \mb{R}$ for some neighborhood $U_0$ of identity, such that $x_i(\tx{exp}(\sum_j a_j e_j)) = a_i$. We extend $x_i$ to $G$ by requiring that $x_i = 1$ outside some compact set containing $U_0$. Then we have the following Lie-theoretic central limit theorem:

\begin{theorem}[Wehn's Theorem, as presented in \cite{stroock1973limit}]
	Suppose we have real numbers $b_j$, $a_{ij}$ such that:
	\begin{itemize}
	\item $\mu_n(G \setminus U) = o(n^{-1})$ for every neighborhood $U$ of identity in $G$.
	\item $\int x_i(g) d \mu_n(g) = \frac{b_i}{n} + o(n^{-1})$
	\item $\int x_i(g) x_j(g) d \mu_n(g) = \frac{a_{ij}}{n} + o(n^{-1})$
	\end{itemize}
	Then there exists a continuous semigroup of probability measures $(\nu_t)_{0 \le t \le 1}$ on $G$, specified by the infinitesimal generator $L \nu = \sum_i b_i e_i + \frac{1}{2} \sum_{ij} a_{ij} e_i e_j$ such that, if $(X_t)_{0 \le t \le 1}$ is the random process whose law is $(\nu_t)$, then $(X_t^n)$ converges weakly to $X_t$ as $n \to \infty$.
\end{theorem}

Note that $G$ being a connected Lie group is necessary for ensuring, for instance, that a continuous semigroup of probability measures with the specific infinitesimal generator exists on $G$.
Now we focus on the case $G = \tx{Sp}(2, \mb{R})$. Let $\tx{sp}(2, \mb{R})$ denote the associated Lie algebra, where we select the following basis:
\[e_1 = \pmtx{1 & 0 \\ 0 & -1}, e_2 = \pmtx{0 & 1 \\ 0 & 0}, e_3 = \pmtx{0 & 0 \\ 1 & 0}\]
Let $X$ be a Gaussian random variable in the 3-dimensional vector space $\tx{sp}(2, \mb{R})$, centered at $(0, 0, 0)$ with covariance matrix $\frac{c}{N} \tx{Id}$ for some parameters $c \in \mb{R}^+, N \in \mb{N}^+$. Then $\tx{exp}(X)$ is a  $\tx{Sp}(2, \mb{R})$-valued random variable. Let $X_1, \ldots, X_N$ be iid variables that have the law of $\tx{exp}(X)$, and define the process $(X_t^N)$ analogously as $(X_t^n)$ above. By Wehn's theorem:

\begin{corollary}
	As $N \to \infty$, $(X_t^N)$ converges weakly to $(X_t)$, a stochastic process associated with the infinitesimal generator $\frac{c}{2} \sum_{i = 1}^{3} e_i^2$.
	\label{corollary:lie-random-walk}
\end{corollary}

By a parallel with the real case, we refer to such process as the \ti{Brownian motion} of parameter $c$ on $G$. Notice that $(X_t)$ is a Levy process, and thus $(X_t)$ are almost surely continuous paths in $t$; we refer to them as \ti{Brownian paths}. Due to the particular form of the generator (c.f. \cite{breuillard2007random}), the probability measures $(\nu_t)$ are absolutely continuous with respect to the Haar measure of $G$. For the canonical choice of the Haar measure, $\tx{Sp}_0$ has measure zero, so it is almost always true that $X_1$ does not land on it. In other words, the random path $t \mapsto X_t$ almost always has a well-defined Conley-Zehnder index, and thus $(\nu_t)$ gives a distribution of the index. Now we can state our main result:

\begin{theorem}
	\label{thm:main}
	Let $M_k$ be the $k$-th moment of the said Conley-Zehnder index distribution. Then we have $M_{2i + 1} = 0$ and $M_{2i} = \Theta(c^{i})$ for all $i \in \mb{N}$.
\end{theorem}

Notice that this is analogous to the situation of simple random walk on $\mb{R}$: if we let $X_1, \ldots, X_N$ be iid normal variables with expectation $0$ and variance $\frac{c}{N}$, and let $S_j = X_1 + \ldots + X_j$, then the odd moments of $S_N$ vanish and the $(2k)$-th moment of $S_N$ is $c^k$. In the last section, we saw that the Conley-Zehnder index can be computed by first projecting to a path on the unit circle then computing its degree, so the ``spirit'' of the statement above is that Brownian motion on $\tx{Sp}(2, \mb{R})$ ``translates down'' to that on $S^1$.

\begin{remark}
	Numerical simulations show that the same holds for general $\tx{Sp}(2n, \mb{R})$. In the next version of the paper we plan to upgrade our argument to account for $n > 1$. See Appendix~\ref{app:simulation} for more details.
\end{remark}

\section{Translation to Random Matrix}
\label{sec:random-matrix}
To make the computation easier, we replace distribution $\tx{exp}(X)$ with a simpler distribution; for the motivation of this choice, see the end of this section. Let $X, Y, Z$ be Gaussian variables in $\mb{R}$ centered around $0$ with variance $\frac{c}{N}$. The following expression induces a probability distribution on $\tx{Sp}(2, \mb{R})$, where we set the probability to be zero for any Borel set contained in the complement of the image:
\[\pmtx{e^X & e^X Y \\ e^X Z & e^{-X} + e^X Y Z}\]
Let $Y_1, \ldots, Y_N$ be iid variables following such distribution, and define the discrete random walk $(Y^N_t)$ analogously. First we want to say that this is sampling from a neighborhood of the identity.

\begin{lemma}
	As $N \to \infty$, it is almost always true that $Y_1, \ldots, Y_N \in \im(\exp)$.
\end{lemma}

\begin{proof}
	Recall that the condition for being in the image of the exponential map is that $\tx{tr} > -2$ or being equal to $-I$. The former is guaranteed if we have both $e^X < 2 \Leftrightarrow X < \log(2)$ and $Y Z > -2$, which is in turn guaranteed if we have $|X| < \log(2), |Y|, |Z| < \sqrt{2}$. The chance of these being maintained during all $N$ rounds is $\left(\text{erf}\left(\frac{1}{\sqrt{\frac{c}{N}}}\right)^2 \text{erf}\left(\frac{\log (2)}{\sqrt{2} \sqrt{\frac{c}{N}}}\right)\right)^N \ge \text{erf}\left(\frac{\sqrt{N}}{\sqrt{8 c}}\right)^{3 N} = \Theta(\tx{erf}(\sqrt{N}))^{3N} \to 1$ as $N \to \infty$.
\end{proof}

\begin{lemma}
	As $N \to \infty$, $(Y^N_t)$ also converges to $(X_t)$ for the same process as in Corollary~\ref{corollary:lie-random-walk}.
\end{lemma}
\begin{proof}
	Let $A = \pmtx{e^X & e^X Y \\ e^X Z & e^{-X} + e^X Y Z}$. Directly Taylor expand $\tx{log}(A) = (A - I) - \frac{(A - I)^2}{2} + \frac{(A - I)^3}{3} + \ldots$ and term-by-term Taylor expand on $X, Y, Z$, we get that, if we write $\log(A) = \sum_{j} a_i e_i$, then $\mb{E}(a_i) = 0 + o(N^{-1})$, $\mb{E}(a_i^2) = \frac{c}{N} + o(N^{-1})$ and $\mb{E}(a_i a_j) = 0$ for $i \neq j$. Now invoke Wehn's theorem.
\end{proof}

By Wehn's theorem, to sample a random path $(X_t)$ and compute its corresponding discrete isotopies is equivalent to directly sampling $N$ matrices with the said distribution, as $N \to \infty$. Suppose these matrices are $Y_i = \pmtx{p_i & q_i \\ r_i & s_i}$ for $i \in [0, N - 1]$, and that they yield generating functions $V_i(x, y) = a_i x^2 + b_i xy + d_i y^2$, where $a_i = \frac{-r_i}{2p_i}$, $b_i = 1 - p_i^{-1}$ and $d_i = \frac{q_i}{2p_i}$. Let us assume that their product does not land on $\tx{Sp}_0$, which is almost always the case. Then we can directly write out $d^2 \Phi(z)$ as $\pmtx{A & B \\ B^t & D}$, where
\[A = 
\pmtx{0 \\
		& -a_0 \\
		& & -a_1 \\
		& & & \ldots \\
		& & & & -a_{N - 1}
}
\hspace{1em}
B =
\pmtx{-1 \\
	1 - b_0 & -1 \\
	& 1 - b_1 & -1 \\
	& & \ldots & \ldots \\
	& & & 1 - b_{N - 2} & -1 \\
	& & & & 1 - b_{N - 1}
	}
\]
\[D = 
\pmtx{-d_0 \\
		& -d_1 \\
		& & -d_2 \\
		& & & \ldots \\
		& & & & -d_{N - 1}
}\]
Let us first turn it into a tridiagonal form: note that if $(L, (x_0, \ldots, x_N, y_0 \ldots, y_{N-1}))$ is an eigenpair of the matrix, then $(L, (x_0, y_0, \ldots, x_{N - 1}, y_{N - 1}, x_N, y_N = 0))$ is an eigenpair of the following matrix (this introduces a zero eigenvalue with eigenvector $(0, \ldots, 0, 1)$, which doesn't affect the signature and which we won't look at anyways):

\[\pmtx{
0 & -1 & \\
-1 & -d_0 & 1 - b_0 \\
   & 1 - b_0 & -a_0 & -1 \\
   &  & -1 & -d_1 & 1 - b_1 \\
   &  &    & \ldots & \ldots & \ldots \\
   &  &  &  & -1 & -d_{N - 1} &  1 - b_{N - 1} \\
   &  &  &  & & 1 - b_{N - 1} & -a_{N - 1} & -1 \\
   &  &  &  & &  & 0 & 0 \\
}\]

Do the conjugation $A \mapsto K^{-1} A K$, $K = \tx{diag}(k_1, \ldots, k_{2N+2})$ where $k_{2j-1} / k_{2j} = 1$, $k_{2j} / k_{2j + 1} = \frac{-1}{(1 - b_{j - 1})}$, we get the following matrix:

\[\pmtx{
0 & -1 & \\
-1 & -d_0 & -(1 - b_0)^2 \\
   & -1 & -a_0 & -1 \\
   &  & -1 & -d_1 & -(1 - b_1)^2 \\
   &  &    & \ldots & \ldots & \ldots \\
   &  &  &  & -1 & -d_{N - 1} &  -(1 - b_{N - 1})^2 \\
   &  &  &  & & -1 & -a_{N - 1} & -1 \\
   &  &  &  & &  & 0 & 0 \\
}\]
Now if we plug in the $a_i, b_i, d_i$ using the alternative distribution defined above, we get the following:

\[\pmtx{
0 & -1 & \\
-1 & -Y_0/2 & -e^{-2 X_0} \\
   & -1 & Z_0/2 & -1 \\
   &  & -1 & -Y_1/2 & -e^{-2 X_1} \\
   &  &    & \ldots & \ldots & \ldots \\
   &  &  &  & -1 & -Y_{N - 1}/2 &  -e^{-2 X_{N - 1}} \\
   &  &  &  & & -1 & Z_{N - 1}/2 & -1 \\
   &  &  &  & &  & 0 & 0 \\
}\]
where $X_i, Y_i, Z_i$ are Gaussian variables of expectation $0$ and variance $\frac{c}{N}$. This particularly simple form is what we aimed at when choosing the alternative distribution---if we had used $\tx{exp}(X)$ the expressions would be much messier. Thus it suffices to understand the signature of this random matrix.

\section{Translation to SDE}
\label{sec:random-sde}
In \cite{valko2009continuum}, the authors studied eigenvalue distributions of a tridiagonal matrix by converting it into a problem of studying the evolution of a single point under random Mobius transformations, which in turn was described by a stochastic differential equation as the continuous limit. We follow this idea to study our current problem.

That $(L, (x_0, y_0, \ldots, x_N, y_N = 0))$ is an eigenpair means that, for $0 \le i < N$:
\[L x_0 = - y_0\]
\[L x_{i + 1} = -y_i - y_{i + 1} +(Z_{i}/2) x_{i+1}\]
\[L y_i = - x_i - (Y_i/2) y_i - e^{-2 X_i} x_{i + 1}\]
\[L y_N = 0\]
Introduce $r_i = y_i / x_i, s_i = x_{i + 1} / y_i \in \overline{\mb{R}}$, where $\overline{\mb{R}}$ denotes the compactified real line. (Notice that eigenvector being nonzero forbids the occurrence of $0 / 0$.) Then we get:
\[r_0 = -L\]
\[L = Z_{i - 1}/2 - s_{i - 1}^{-1} - r_i\]
\[L = -Y_i/2 - r_i^{-1} - e^{-2 X_i} s_i\]
\[r_N = 0\]
which gives the following evolution of $r_i$:
\[r_{i + 1} = \frac{(2 e^{-2X} - \frac{1}{2}(2L+Y)(2L-Z)) r_i - (2L-Z)}{(2L+Y) r_i + 2}\]
Here we suppress---and will do so from now on---the indices of the random variables, with the understanding that each step a new set of random variables is being used. Then the study turns into studying for which $L$ do we have the boundary conditions (i.e. $r_0 = -L, r_N = 0$) met.
Let $\tilde{f}_L(u) = \tilde{f}(L, u)$ denote the single-step evolution:
\[u \mapsto \frac{(2 e^{-2X} - \frac{1}{2}(2L+Y)(2L-Z)) u - (2L-Z)}{(2L+Y) u + 2}\]
This is a Mobius transformation which is a perturbed version of the following transformation $f_L(u) = f(L, u)$:
\[u \mapsto \frac{(1 - L^2) u - L}{L u + 1}\]

Since we're interested in the signature it suffices to count the number of positive eigenvalues, so for what follows we assume $L \ge 0$. Let $\mc{J}: \overline{\mb{R}} \to S^1 \subset \mb{C}$ be the map from the extended real to the unit circle in the complex plane, given by $x \mapsto \frac{i - x}{i + x}$, so the inverse is given by $\mc{J}^{-1}(e^{i \theta}) = \tan(\theta / 2)$. Let $\mc{R}$ denote the universal covering of $S^1 \cong \overline{\mb{R}}$, where we fix a lift of $0$ to the basepoint of $\mc{R}$. Given any Mobius transformation $f$, $\mc{J} f \mc{J}^{-1}$ can be lifted to a continuous function $\mc{R} \to \mc{R}$, which we'll denote by $f^\circ$. In particular, consider $\tilde{f}_L^\circ$ for $L \ge 0$ and its $k$th iterations $(\tilde{f}_L^\circ)^{(k)}$, each time with a different set of random variables. Define a continuous function $\mc{F}(k, L): \mb{N} \times [0, \infty] \to \mc{R}$ as $(\tilde{f}_L^\circ)^{(k)}(\mc{J}(-L))$, then counting how many positive eigenvalues there are is the same as counting how many times $\mc{F}(N-1, L)$ crosses the $(2 j \pi)_{j \in \mb{Z}}$ lines for $L \in [0, \infty]$. Thus by counting these crossings we can compute the signature, which stabilizes as we take $N \to \infty$. The following shows that it suffices to look at the two ends: $L = 0$ and $L \to \infty$.

\begin{lemma}
	$\mc{F}(k, L)$ monotonically decreases in $L$ for any $k$ and values of random variables.
\end{lemma}

\begin{proof}
	$\mc{J}(-L)$ decreases monotonically from $0$ to $-\pi$, so it suffices to prove that, given $u'(L) < 0$ for all $L \ge 0$, $\frac{\partial}{\partial L} \tilde{f}_L^\circ(u(L)) < 0$ for all $L \ge 0$.
Notice that
\[\frac{\partial} {\partial L}\tilde{f}_L^\circ(u(L)) = \frac{\partial}{\partial L} 2 \arctan(\tilde{f}_L(\tan(\frac{u(L)}{2}))) = \frac{\sec(\frac{u(L)}{2})^2 u'(L) \tilde{f}_L'(\tan(\frac{u(L)}{2}))}{1 + \tilde{f}_L(\tan(\frac{u(L)}{2}))^2}\]
thus it suffices to prove that $\tilde{f}_L'(\tan(\frac{u(L)}{2})) > 0$. In fact, $\tilde{f}_L'(x) > 0$ for all $x$: from the alternative way of writing $\tilde{f}_L(x) = - L + Z/2 + \frac{1}{e^{2 X}(L + Y / 2 + 1/x)}$ we see that the values of $X, Y, Z, L$ do not affect the sign of the derivative, and if we set all of these to be $0$ the derivative is simply $1$.
% just take the derivative
\end{proof}

The $\infty$ end is uninteresting, as the following lemma shows:

\begin{lemma}
	For any values of random variables and any $k$, $\lim_{L \to \infty} (\tilde{f}_L^\circ)^{(k)}(\mc{J}(-L)) = -(2k + 1) \pi$.
\end{lemma}

\begin{proof}
	$\lim_{L \to \infty} \mc{J}(-L) = - \pi$, so it suffices to prove that $\tilde{f}_L^\circ(t) = t - 2 \pi$ when $L$ goes to infinity assuming $t = - \pi \pmod{2 \pi}$. By linearity it suffices to prove this for $t = -\pi$. When $L$ goes to infinity, the Mobius transformation tends to the same limit as the noiseless version. So let $f^\circ_L$ denote the lift of the unperturbed single-step evolution. Notice that $\frac{\partial}{\partial L}f_L^\circ(t) = -\frac{-2 L^2 \cos (t)+2 L^2+4 L \sin (t)+4}{L^4+2 L^3 \sin (t)-\left(L^2-2\right) L^2 \cos (t)+2}$, which equals $-\frac{2 \left(L^2+1\right)}{L^4-L^2+1}$ for $t = -\pi$. We know $f_0^\circ$ is the identity, so we have $f_L^\circ(-\pi) - (-\pi) = \int_{L = 0}^{\infty} -\frac{2 \left(L^2+1\right)}{L^4-L^2+1} dL = - 2 \pi$.
\end{proof}

%two random variables have a deterministic relationship between each other
Therefore we can relate the signature of the Hessian to the value of $\mc{F}(N - 1, 0)$. Since $d^2 \Phi(z)$ is a real symmetric matrix, it has exactly $2N - 1$ distinct eigenvalues. By monotonicity, the number of positive eigenvalues is equal to the number of multiples of $2 \pi$ between $\mc{F}(N - 1, 0)$ and $-(2N-1) \pi$, which is $N - 1 - \lfloor \frac{-\mc{F}(N - 1, 0)}{2 \pi} \rfloor$. Since the Hessian is nondegenerate, we have $\tx{sign}~d^2 \Phi(z) = 2(N - 1 - \lfloor \frac{-\mc{F}(N - 1, 0)}{2 \pi} \rfloor) - (2N - 1) = - 2 \lfloor \frac{-\mc{F}(N - 1, 0)}{2 \pi} \rfloor - 1$. This deterministic relationship between the two random variables $\tx{sign}~d^2 \Phi(z)$ and $\mc{F}(N - 1, 0)$ allows us to translate the study of the probability distribution of the former to that of the latter.

At $L = 0$, the noiseless transformation is the identity so the effect is entirely that of the random noise. The evolution is given by repeatedly applying $\left(
\begin{array}{cc}
 \frac{Y Z}{2}+2 e^{-2 X} & Z \\
 Y & 2 \\
\end{array}
\right)$ to $0$. In the angular form, we have
\[\tilde{f}_0^\circ(t) = 2 \tan ^{-1}\left(\frac{\tan \left(\frac{t}{2}\right) \left(2 e^{-2 X}+\frac{Y Z}{2}\right)+Z}{2+Y \tan \left(\frac{t}{2}\right)}\right)\]
Now do Taylor expansion on $X, Y, Z$ to obtain the following one-step evolution estimate:
\[\mb{E}(\tilde{f}_0^\circ(t) - t) = \frac{7 c}{8 n} \sin(2 t) + o(n^{-1}) \hspace{1em} \mb{E}((\tilde{f}_0^\circ(t) - t)^2) = \frac{c}{4 n}(11 - 7 \cos(2 t)) + o(n^{-1})\]
which shows that the value varies slow enough such that, if we take $N \to \infty$, the discrete Markov chain converges to a continuous Ito diffusion process, as the following statement describes:

\begin{proposition}[\cite{valko2009continuum}, Proposition 23]
	Fix $T > 0$, and for each $n \ge 1$ let $(X_\ell^n \in \mb{R}^d, \ell = 1, \ldots, \lfloor n T \rfloor)$ be a Markov chain. Let $Y_\ell^n(x)$ be distributed as $X_{\ell+1}^n - X_{\ell}^n$ given that $X_\ell^n = x$. Define $b^n(t, x) = n \mb{E}(Y^n_{\lfloor n t \rfloor}(x))$ and $a^n(t, x) = n \mb{E}((Y^n_{\lfloor n t \rfloor}(x))^t (Y^n_{\lfloor n t \rfloor}(x)))$. Suppose that as $n \to \infty$, the following conditions hold for some $M$:
	\[|a^n(t, x) - a^n(t, y)| + |b^n(t, x) - b^n(t, y)| \le M|x - y| + o(1)\]
	\[\sup_{x, \ell} \mb{E}(|Y_\ell^n(x)|^3) \le M n^{-3/2}\]
	and that there exists some $a, b: \mb{R} \times [0, T] \to \mb{R}$ with bounded first and second derivatives such that
	\[\sup_{x, t}\left|\int_{0}^{t} a^n(s, x)ds - \int_{0}^{t} a(s, x)ds\right| + \sup_{x, t}\left|\int_{0}^{t} b^n(s, x)ds - \int_{0}^{t} b(s, x)ds\right| \to 0\]
	and that $X_0^{n}$ converges weakly to some distribution $X_0$, then $(X_{\lfloor n t \rfloor}^n, 0 \le t \le T)$ converges in law to the unique strong solution of the SDE (where $\sigma^t \sigma = a$)\footnote{The original paper had a typo and put $a$ in place of $\sigma$.}:
	\[dX_t = b(t, X_t) dt + \sigma(t, X_t) dB, X(0) = X_0\]
\end{proposition}

The exact parameters are $b(s, x) = \frac{7c}{8} \sin(2 x)$ and $\sigma(s, x) = \sqrt{\frac{c}{4}(11 - 7 \cos(2 x))}$ and $\theta_0$ is $\delta_0$, the Dirac distribution at $0$. Thus the evolution at $L = 0$ is governed by the distribution $\theta_1$ of the following stochastic differential equation:
\[d \theta_t = \frac{7 c}{8} \sin(2 \theta_t) d t + \sqrt{\frac{c}{4}(11 - 7 \cos(2 \theta_t))} d B\]
Let $\theta'_t = \theta_{t/c}$, then by the time-change formula (c.f. \cite[V.26]{rogers2000diffusions}), $\theta'_t$ satisfies the same equation as above for $c = 1$. Since $\theta_1 = \theta'_c$, to understand the dependence on $c$ is the same as understanding the dependence on $t$ after fixing $c = 1$.

The Fokker-Planck equation, given as follows, describes the evolution of the probability distribution function $p(t, x) = \tx{Pr}(\theta_t = x \mid \theta_0 = 0)$ as the unique solution to the following PDE:

\[\frac{\partial p(t, x)}{\partial t} = \frac{\partial}{\partial x}\left(b(x) p(t, x)\right) + \frac{\partial^2}{\partial x^2}\left(\frac{1}{2} \sigma(x)^2 p(t, x)\right), p(0, x) = \delta_0(x)\]

This parabolic PDE is not easy to solve, but we need not solve it explicitly. Note that $b(x)$ and $\sigma(x)$ are both bounded, with $\sigma(x)$ bounded from below by a positive number. This is enough to guarantee that the solution behaves basically like a heat diffusion. The necessary estimate was established by Norris and Stroock in \cite{norris1991estimates}. We use the following presentation by Arturo Kohatsu-Higa that is tailored to SDEs:

\begin{theorem}[\cite{kohatsu2003lower}]
	Suppose $X_t \in \mb{R}^d$ is the (strong) solution of the SDE $d X_t = b(t, X_t) dt + \sigma(t, X_t) dB$ for timespan $[0, T]$, and suppose that the measurable coefficient functions $b, \sigma$ satisfy the following uniform ellipticity requirement:
	\begin{itemize}
	\item $\sup_{t \in [0, T]}\left(\sup_{x \in \mb{R}}\|b(t, x)\| + \sup_{x \in \mb{R}}\|\sigma(t, x)\|\right) < \infty$
	\item There exists some $C \in \mb{R}$ such that $\sigma(t, x)^t \sigma(t, x) \ge C$ for all $t, x$.
	\end{itemize}
	Then $p(t, x)$, as described above, satisfies the following estimate for some $m, M \in [1, +\infty)$:
	\begin{equation}
	\frac{m \exp(-\frac{x^2}{m t})}{t^{1/2}} \ge p(t, x) \ge \frac{\exp(-M \frac{x^2}{t})}{M t^{1/2}}
	\end{equation}
\end{theorem}

Thus we immediately obtain that $\mb{E}(\theta_t^{2k}) = \Theta(t^k)$. For odd moments, notice that $p(t, -x)$ is another solution to the Fokker-Planck equation, so by uniqueness we have $p(t, -x) = p(t, x)$; in other words, all odd moments vanish. Theorem~\ref{thm:main} now follows Theorem~\ref{thm:morse-index}.

\newpage

\appendix

\section{Numerical Simulation}
\label{app:simulation}

As part of the initial attempt to understand the moments, we implemented a program in \ti{Mathematica} that allows us to generate discretizations of random paths in $\tx{Sp}(2n, \mb{R})$ and compute their Conley-Zehnder indices. The program generates random walks on the Lie group using the exponential Gaussian measure described in Section~\ref{sec:random-walk}, then explicitly computes a path extension to $W_{\pm}$ (as described in Section~\ref{sec:CZ-index}), applies the $\rho^2$ map to obtain a discrete sequence of points on $S^1$, and finally counts the winding number of this sequence.

The computation of path extension relies heavily on the fact that the ending point of the path is a symplectic matrix, and thus numerical precision becomes crucial. A naive repeated multiplication of matrices would cause loss of numerical precision as the entries of the matrix become large. This can be solved by decomposing a long path into the product of many small segments using the homotopy invariance of the index, and combine their Conley-Zehnder indices using the following product formula:

\begin{lemma}[\cite{de2008product}]
	Let $\Psi_1, \Psi_2: [0, 1] \to \tx{Sp}(2n, \mb{R})$ such that $\Psi_1(0) = \Psi_2(0) = I, \det(\Psi_1(1) - I) \neq 0, \det(\Psi_2(1) - I) \neq 0$. Define $\Psi(t) = \Psi_1(t) \Psi_2(t)$, then we have
	\[\mu_{\tx{CZ}}(\Psi) = \mu_{\tx{CZ}}(\Psi_1) + \mu_{\tx{CZ}}(\Psi_2) - \frac{1}{2} \tx{sign}(\mc{M}(\Psi_1(1)) + \mc{M}(\Psi_2(1)))\]
	where $\tx{sign}$ is the signature, and $\mc{M}$ is given by $\mc{M}(X) = \frac{1}{2} J_0 (X_0 + I) (X_0 - I)^{-1}$.
\end{lemma}

For demonstrative purpose, we fix $n = 3$ and $N = 5000$ and vary $c$, and obtain the following data for moments. The definition of $M_i$ is the same as in Theorem~\ref{thm:main}. The following data table (generated with 500 trials in each case) along with the graphs for $M_2$ and $M_4$ suggest that the asymptotics we proved for $n = 1$ continue to hold for $n > 1$ (of course, odd moments won't vanish for finite data).

\[
\begin{array}{c|cccccccc}
	c & M_1 & M_2 & M_3 & M_4 & M_5 & M_6 & M_7 & M_8 \\ \hline
 5 & -0.096 & 2.104 & -0.684 & 14.272 & -5.916 & 173.824 & -62.844 & 2898.59 \\
 10 & -0.004 & 4.032 & -0.352 & 44.856 & 0.176 & 749.592 & 291.728 & 15739.4 \\
 15 & 0.02 & 5.72 & 0.524 & 92.072 & 5.9 & 2430.44 & -728.596 & 89878. \\
 20 & 0.11 & 7.15 & 4.106 & 150.55 & 187.37 & 4944.07 & 11437.5 & 217410. \\
 25 & -0.062 & 9.258 & -1.718 & 254.97 & -271.862 & 12454. & -41352.8 & 890818. \\
 30 & 0.068 & 12.88 & 11.624 & 481.624 & 837.608 & 27475.2 & 54846.1 & 1.97698\times 10^6 \\
\end{array}
\]

\begin{figure}[H]
\centering
 \subfloat[$M_2$ with linear fit]{\includegraphics[width=0.5\textwidth]{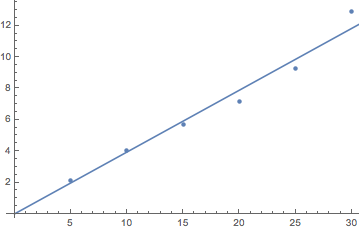}}
  \subfloat[$M_4$ with quadratic fit]{\includegraphics[width=0.5\textwidth]{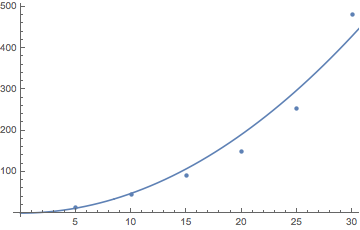}}
\end{figure}

Our proof in the $n = 1$ case actually yields a better description than asymptotics and says that the distribution of Conley-Zehnder indices should approximate normal distribution as $c \to \infty$. The following histograms, generated using the same data as above, support this claim.

\begin{figure}[H]
  \centering   
  \subfloat[$c = 5$]{\includegraphics[width=0.3\textwidth]{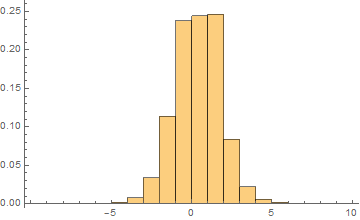}}
  \subfloat[$c = 10$]{\includegraphics[width=0.3\textwidth]{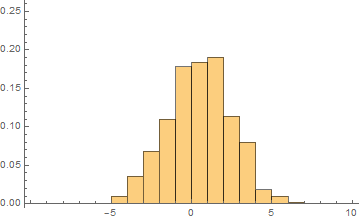}}
  \subfloat[$c = 15$]{\includegraphics[width=0.3\textwidth]{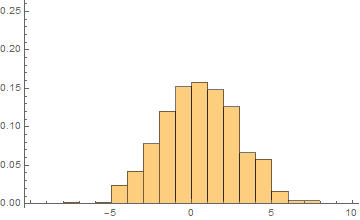}} \\
  \subfloat[$c = 20$]{\includegraphics[width=0.3\textwidth]{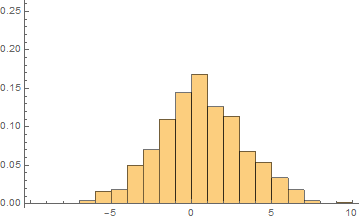}}
  \subfloat[$c = 25$]{\includegraphics[width=0.3\textwidth]{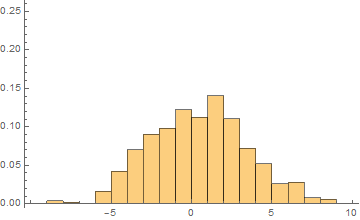}}
  \subfloat[$c = 30$]{\includegraphics[width=0.3\textwidth]{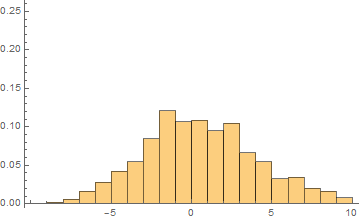}}
\end{figure}

\nocite{*}
\printbibliography[heading=bibliography]

\end{document}